\documentclass[10pt,fleqn]{amsart}

\usepackage{amsmath,amssymb,latexsym}
\usepackage[mathscr]{eucal}

\theoremstyle{plain}
\newtheorem{theorem}{Theorem}
\newtheorem{lemma}[theorem]{Lemma}

\newcounter{rmkctr}

\numberwithin{equation}{section}
\allowdisplaybreaks
\newcommand{\dy}{\partial}

\newcommand{\ddt}[1]{\frac{\mathrm{d}{#1}}{\mathrm{d}{t}}}

\newcommand{\sfrac}[2]{{\textstyle\frac{#1}{#2}}}

\newcommand{\Zahl}{\mathbb{Z}}

\newcommand{\ex}{\mathrm{e}}
\newcommand{\im}{\mathrm{i}}
\newcommand{\eps}{\varepsilon}

\newcommand{\eikx}{\ex^{\im\kb\cdot\xb}}

\newcommand{\dtau}{\>\mathrm{d}\tau}
\newcommand{\dx}{\>\mathrm{d}\boldsymbol{x}}

\newcommand{\lapl}{\Delta}
\newcommand{\ilapl}{\Delta^{-1}}

\newcommand{\xb}{{\boldsymbol{x}}}
\newcommand{\vb}{{\boldsymbol{v}}}

\newcommand{\fv}{f_{\vb}^{}}

\newcommand{\kb}{{\boldsymbol{k}}}

\newcommand{\gb}{\nabla}

\newcommand{\divv}{\gb\!\cdot\!\vb}

\newcommand{\Proj}{{\sf P}}
\newcommand{\Pb}{\bar{\sf P}}
\newcommand{\Pt}{\tilde{\sf P}}

\newcommand{\Dom}{\mathscr{M}}
\newcommand{\ZL}{\Zahl_L}

\newcommand{\Attr}{\mathcal{A}}
\newcommand{\DN}{\mathcal{E}}

\newcommand{\w}{\omega}
\newcommand{\wb}{\bar{\w}}
\newcommand{\wt}{\tilde{\w}}

\newcommand{\pt}{\tilde{\psi}}
\newcommand{\fb}{\bar{f}}
\newcommand{\ft}{\tilde{f}}

\newcommand{\sst}[1]{{\scriptscriptstyle{#1}}}

\newcommand{\dw}{\delta\omega}

\newcommand{\dwt}{\delta\tilde{\w}}

\newcommand{\dps}{\delta\psi}
\newcommand{\dpsm}[1]{\dps^{\scriptscriptstyle{#1}}}

\newcommand{\dpst}{\delta\tilde{\psi}}

\newcommand{\dpstm}[1]{\dpst^{\scriptscriptstyle{#1}}}

\newcommand{\dwm}[1]{\dw^{\scriptscriptstyle{#1}}}

\newcommand{\dwtm}[1]{\dwt^{\scriptscriptstyle{#1}}}

\newcommand{\wfb}[1]{\wb^{\scriptscriptstyle{#1}}}

\newcommand{\ffb}[1]{\fb^{\scriptscriptstyle{#1}}}

\newcommand{\Gr}{\mathcal{G}}
\newcommand{\TB}{\mathcal{T}}

\newcommand{\kpz}{\kappa_0}
\newcommand{\kpf}{\kappa_f}

\newcommand{\ws}{\w^{\sharp}}
\newcommand{\pss}{\psi^{\sharp}}

\begin{document}

\title[Determining modes and nodes of $\beta$-plane NSE]%
{Navier--Stokes equations on the $\beta$-plane:\\
determining modes and nodes}

\author[Miyajima]{N.~Miyajima}
\email{naoko.miyajima@durham.ac.uk}
\author[Wirosoetisno]{D.~Wirosoetisno}
\email{djoko.wirosoetisno@durham.ac.uk}
\urladdr{http://www.maths.dur.ac.uk/\~{}dma0dw}
\address{Department of Mathematical Sciences\\
   University of Durham\\
   Durham\ \ DH1~3LE, United Kingdom}



\begin{abstract}
We revisit the 2d Navier--Stokes equations on the periodic $\beta$-plane, with the Coriolis parameter varying as $\beta y$, and obtain bounds on the number of determining modes and nodes of the flow.
The number of modes {and nodes} scale as $c\,\Gr_0^{1/2} + c'(M/\beta)^{1/2}$ and $c\,\Gr_0^{2/3} + c'(M/\beta)^{1/2}$ respectively, where the Grashof number $\Gr_0=|f_\vb|_{L^2}^{}/(\mu^2\kpz^2)$ and $M$ involves higher derivatives of the forcing $f_\vb$.
For large $\beta$ (strong rotation), this results in fewer degrees of freedom than the classical (non-rotating) bound that scales as $c\,\Gr_0$.
\end{abstract}

\maketitle


\section{Introduction}\label{s:intro}

Understanding the behaviour of rotating fluid flows is fundamental to many problems in geophysical fluid dynamics.
The simplest rotating fluid model is arguably the 2d Navier--Stokes equations, which however is unaffected by constant (rigid body) rotation.
It is affected, however, by differential rotation, such as that in a rotating sphere or its simplified model, the $\beta$-plane.
In this case one expects on physical grounds that the flow will become more zonal (i.e.\ less dependent on the ``longitude'' $x$) as the rotation rate increases.

To quantify this, we decompose the (scalar) vorticity as $\omega(x,y,t)=\wb(y,t)+\wt(x,y,t)$, with the zonal part $\wb$ obtained by averaging $\omega$ over $x$.
In \cite{mah-dw:beta} and \cite{dw:nss2}, it was proved that the non-zonal part of the flow becomes small as $t\to\infty$, in the sense that $|\wt(t)|_{L^2}^2\le \eps M_0$ for sufficiently large $t$.
It was also proved that the global attractor $\Attr$ reduces to a point for $\eps$ sufficiently small (but still finite).
Naturally, this begs the question of how the number of degrees of freedom in the flow scales with $\eps$.
In the non-rotating case, the results on determining modes and attractor dimensions agreed (essentially, up to a logarithm) with those expected on physical grounds from the Kolmogorov theory, after two decades of effort \cite{foias-prodi:67,constantin-foias-temam:88,jones-titi:93}.

The present rotating case is more delicate, and there is as yet no physical consensus on the number of degrees of freedom as a function of $\eps$: as discussed in \cite[\S9.1.1]{vallis:aofd}, there are several plausible estimates of the Rhines wavenumber $\kappa_\beta$, roughly the smallest wavenumber (largest scale) that supports turbulent flows \cite{rhines:75,vallis-maltrud:93}.
These physical estimates depend only on the energy $|\vb|_{L^2}^2$ and enstrophy $|\omega|_{L^2}^2$, although arguably the arguments implicitly assume certain unspecified smoothness of the flows.

Extending the results from \cite{mah-dw:beta}, and using tools from \cite{jones-titi:92n,jones-titi:93}, in this paper we prove bounds on the number of determining modes and nodes related to the number of degrees of freedom in the rotating NSE.
Unlike the physical estimates in the previous paragraph, our rigorous results inevitably involve higher derivatives of the vorticity (and thus the forcing).
It is not clear at this point whether our bounds are optimal, particularly as one does not know what to expect on physical grounds.

A natural extension of our results is to bound the Hausdorff dimension of the global attractor $\Attr$.
This we have not been able to do, and it appears that current methods to estimate attractor dimensions (e.g., \cite{robinson:idds,doering-gibbon:aanse,ilyin-titi:08}) are not directly applicable to our problem.
Given a bound on the attractor dimension, an analogous bound on the number of determining nodes would follow from \cite{friz-robinson:01}: if $N>32\,\hbox{dim}_H\Attr$, then almost every set of $N$ nodes is determining.
We are not however aware of any result in the opposite direction (which is what is needed in our case).

We expect that our results could be extended to the more realistic case of the rotating sphere with minimal additional conceptual difficulty; cf.~\cite{dw:nss2}.
However, as the bounds obtained here may not be optimal, we do not do so in this paper.

\medskip\hbox to\hsize{\qquad\hrulefill\qquad}\medskip

We consider the two-dimensional rotating Navier--Stokes equations in the so-called $\beta$-plane approximation,
\begin{equation}\label{q:dvdt}\begin{aligned}
   &\dy_t \vb + \vb \cdot \gb \vb + \beta y \vb^{\perp} + \gb p = \mu \lapl \vb + \fv,\\
   &\divv = 0.
\end{aligned}\end{equation}
Here $\vb = (v_1, v_2)$ is the velocity of the fluid, $p$ is the pressure, $\mu$ is the kinematic viscosity and $\fv$ is the forcing on the velocity, assumed to be independent of time.
The term $\beta y\vb^\perp$, where $\vb^{\perp} := (-v_2, v_1)$, arises from the differentially rotating frame, which can be thought of as a linearised approximation of a region on a rotating sphere.
We take as our domain $\Dom = [0,L] \times [-L/2,L/2]$ with periodicity in both directions assumed.
We assume without loss of generality that
\begin{equation}\label{q:v0int}
   \int_{\Dom} \vb \dx= 0.
\end{equation}
For consistency with the periodic domain, we also assume the following symmetries:
\begin{equation}\label{q:velsym}\begin{aligned}
   &v_1(x,-y,t) = v_1(x,y,t),\\
   &v_2(x,-y,t) = -v_2(x,y,t),
\end{aligned}\end{equation}
with analogous symmetries imposed on $\fv$.

We drop all dimensions except length, so $\vb$ and $\fv$ have dimensions of length, $\gb$ has dimension (length)$^{-1}$ and $\mu$ has dimension (length)$^2$; the $L^p$ norm $|\cdot|_{L^p(\Dom)}$ has dimension (length)$^{2/p}$, with $|\cdot|_{L^\infty}^{}$ being naturally dimensionless.
Constants denoted by $c$ and numbered constants $c_i$ are dimensionless.

With this non-dimensionalisation, we take $\gb^\perp\cdot{}$(\ref{q:dvdt}a) to get
\begin{equation}\label{q:dwdt}
   \dy_t \w + \dy(\psi,\w) + \frac{\kpz}{\eps}\dy_{x}\psi = \mu\lapl\w + f,
\end{equation}
where the (scalar) vorticity is $\w:=\gb^\perp\cdot\vb=\dy_x v_2-\dy_y v_1$, which conveniently is dimensionless.
Here $\dy(\cdot,\cdot)$ denotes the Jacobian, i.e.\ $\dy(f,g):=\dy_xf\,\dy_yg-\dy_xg\,\dy_yf$, which has the property that
\begin{equation}
	(\dy(f,g),g)_{L^2(\Dom)}^{} = 0
\end{equation}
for all $f$, $g$ such that the expression is defined.
The forcing (on vorticity) is $f:=\gb^\perp\cdot\fv$, $\eps\sim1/\beta$ (both dimensionless) and $\kpz=2\pi/L$ is the Poincar\'e constant for $\Dom$.
For later use, we define the (dimensionless) parameter $\nu_0 := \mu\kpz^2$ and assume for convenience that $\nu_0\le1$ (we shall use the fact that $\ex^{\nu_0}<3$ below).

The streamfunction $\psi$ is defined uniquely by
\begin{equation}
   \psi := \ilapl\w \qquad\text{with}\quad \int_\Dom \psi \dx = 0.
\end{equation}
We note that due to the property of the curl,
\begin{equation}\label{q:0int}
   \int_{\Dom} \w \dx= 0.
\end{equation}
Moreover, the symmetries \eqref{q:velsym} imply
\begin{equation}\label{q:wsym}
   \w(x,-y,t) = -\w(x,y,t).
\end{equation}
It follows from the symmetries on $f_{\vb}$ that $f(x,-y,t)=-f(x,y,t)$ for all $\xb$ and $t$.
Thanks to \eqref{q:v0int} and \eqref{q:0int}, the $H^s$ norm is equivalent to
\begin{equation}
  |\gb^s\w|_{L^2}^2 := |(-\Delta)^{s/2}\w|_{L^2}^2.
\end{equation}

It is a classical result that, given $\fv$ and $\vb(0)\in L^2$, the NSE \eqref{q:dwdt} has a globally unique solution that is bounded only by the forcing (i.e.\ independently of the initial data), for sufficiently large times, in terms of the Grashof number
\begin{equation}\label{q:gr}
   \Gr = \frac{|\fv|_{L^2}^{}}{\mu^2\kappa_0^2} =: \Gr_0.
\end{equation}
Defining ``higher Grashof numbers'' by
\begin{equation}
   \Gr_m := \frac{|\gb^m\fv|_{L^2}^{}}{(\mu\kpz)^{2-m}},
\end{equation}
we can bound derivatives of the vorticity independently of the initial data,
\begin{equation}\label{q:wgm}
   |\gb^m\w(t)|_{L^2}^2 + \mu\int_0^t|\gb^{m+1}\w|_{L^2}^2\,\ex^{\nu_0(\tau-t)} \dtau \le c\,(m)\,\frac{\Gr_m^2(1+c'(m)\,\nu_0^2\,\Gr_0^2)^m}{(\mu\kpz)^{2m-2}}
\end{equation}
for all $t\ge T_m(|\vb(0)|_{L^2}^{},|\gb^{m-1}f|_{L^2}^{}; \mu)$ as long as $\Gr_m$ is defined.


\section{Background and Main Results}\label{s:results}

It was discovered fifty years ago \cite{foias-prodi:67} that the solutions of 2d NSE are determined essentially by a finite number of degrees of freedom.
Following Foias and Prodi, we consider two solutions $\w$ and $\w^\sharp$ of \eqref{q:dwdt} with the same $f\in H^{-1}$ but potentially different initial data $\vb(0)$ and $\vb^\sharp(0) \in L^2$,
\begin{equation}\label{q:beta1}
   \dy_t \w + \dy(\psi,\w) + \frac{\kpz}{\eps}\dy_{x}\psi \quad \enskip = \mu\lapl\w + f
\end{equation}
\begin{equation}\label{q:beta2}
   \dy_t \ws + \dy(\pss,\ws) + \frac{\kpz}{\eps}\dy_x \pss = \mu\lapl\ws + f,
\end{equation}
and note that their difference $\dw:=\w-\w^\sharp$ satisfies
\begin{equation}\label{q:dwbeta}
   \dy_t\dw + \dy(\pss,\dw) + \dy(\dps,\w) + \frac{\kpz}{\eps}\dy_x \dps = \mu\lapl\dw.
\end{equation}

We expand $\dw$ in Fourier series,
\begin{equation}
   \dw(\xb,t) = \sum_{\kb\in\ZL}\,\delta\w_\kb(t)\,\ex^{\im\kb\cdot\xb}
\end{equation}
where $\ZL:=\{(2\pi l_1/L,2\pi l_2/L):(l_1,l_2)\in\Zahl^2\}$.
All wavenumber sums, unless otherwise stated, are henceforth understood to be over $\ZL$.
Introducing a threshold wavenumber $\kappa$, we define the $L^2$ projection $\Proj_\kappa$ and
\begin{align}
   \dwm<(\xb,t) &:= \Proj_\kappa\dw(\xb,t) & &:= \sum_{|\kb|\leq\kappa}\,\delta\w_\kb(t) \eikx,\hbox to 70pt{}\\
   \dwm>(\xb,t) &:= \dw(\xb,t) - \dwm<(\xb,t) & &\,= \sum_{|\kb|>\kappa}\,\delta\w_\kb(t) \eikx.
\end{align}

The central idea is that if one takes $\kappa$ sufficiently large, the behaviour of the NSE in the long-time limit is determined only by the low ``determining'' modes (i.e.\ $\Proj_{\kappa}\,\w$), in the sense that if $|\Proj_\kappa\dw(t)|_{L^2(\Dom)}^{}\to0$ as $t\to\infty$, then also $|\dw(t)|_{L^2(\Dom)}^{}\to0$ as $t\to\infty$.
The bound on the number of determining modes was improved considerably in \cite{foias-manley-temam-treve:83}, approaching up to a logarithm what one expects on physical grounds \cite{manley-treve:82}.
Subsequently, Jones and Titi \cite{jones-titi:93} obtained a bound free of the ``spurious'' logarithmic term:

\begin{theorem}\label{t:jt93m}
Let $\dw$ satisfy \eqref{q:dwbeta}.
There exists an absolute constant $c_1$ such that if
\begin{equation}\label{q:jt93m}
   \kappa/\kpz \ge c_1\,\Gr_0^{1/2},
\end{equation}
then
\begin{equation*}
   \lim_{t \to \infty} |\Proj_\kappa\dw(t)|_{L^2(\Dom)}^{} = 0
   \quad \text{implies} \quad
   \lim_{t \to \infty} |\dw(t)|_{L^2(\Dom)}^{} = 0.
\end{equation*}
\end{theorem}

\noindent
We remark that this bound supports the physical intuition that turbulence is extensive, in the sense that if one were to merge two similar systems (having the same dimensions and Grashof numbers), the number of degrees of freedom (viz.\ determining modes), which scales as $(\kappa/\kpz)^2$, would double.

Similarly, following \cite{jones-titi:93} and \cite{foias-temam:84}, we call a set of points $\mathcal{E} = \{\xb_1,\cdots,\xb_N\} \subset \Dom$ determining nodes if
\begin{equation}
	\lim_{t\to\infty} \dw(\xb_i,t) = 0 \quad \forall i \in \{1, \cdots , N\} \quad \text{implies} \quad \lim_{t\to\infty}|\dw(t)|_{L^2(\Dom)}^{} = 0.
\end{equation}
Foias and Temam \cite{foias-temam:84} first proved the existence of such a set and gave a bound on the maximal distance between individual nodes, while Jones and Titi \cite{jones-titi:93} gave the following qualitatively optimal bound on the number of determining nodes.
\begin{theorem}\label{t:jt93n}
Let $\dw$ satisfy \eqref{q:dwbeta}.
There exists an absolute constant $c_2$ and a set of determining nodes $\mathcal{E} = \{\xb_1,\cdots,\xb_N\}$, where
\begin{equation}\label{q:jt93n}
	N \geq c_2 \,\Gr_0,
\end{equation}
i.e.\ $\lim_{t\to\infty}\dw(\xb_i,t) = 0$ for $i \in \{1, \cdots , N\}$ implies that $\lim_{t\to\infty}|\dw(t)|_{L^2(\Dom)}^{} = 0$.
\end{theorem}

The bounds in \eqref{q:jt93m} and \eqref{q:jt93n} are qualitatively equivalent, i.e.\ they involve the same number of degrees of freedom (possibly up to a constant).
They are also independent of the rotation rate $\eps^{-1}$, i.e.\ they hold with or without rotation.
On physical grounds, however, one expects that under a differential rotation, the number of determining modes and nodes would decrease as the rotation rate increases.

To this end, we begin by splitting the vorticity into its zonal (independent of $x$) and non-zonal components,
\begin{equation}
   \wb(y,t) := \frac{1}{L}\int_0^L \w(x,y,t) \>\mathrm{d}x
   \quad\text{and}\quad
   \wt(x,y,t) := \w(x,y,t) - \wb(y,t).
\end{equation}
For convenience, we also define projections to the zonal and non-zonal components,
\begin{equation}
   \Pb\omega := \wb
   \quad\text{and}\quad
   \Pt\omega := (1-\Pb)\,\omega = \wt.
\end{equation}
These are orthogonal projections in $H^m$, commuting with $\Proj_\kappa$.
Moreover, they satisfy
\begin{equation}\label{q:bb0}
	\dy(\rho,\gamma) = 0 \qquad \text{whenever} \qquad \dy_x\rho = \dy_x\gamma = 0.
\end{equation}

The key ingredient for the results in this paper is the bound on the non-zonal component $\wt$ from \cite{mah-dw:beta}.
Here we state it in a form that shows the explicit dependence on $\Gr_m$:
\begin{theorem}\label{t:aw}
Assume that the initial data $\vb(0) \in L^2(\Dom)$ and that $|\lapl f|_{L^2} < \infty$.
Then there exist a $\TB_0(|\vb(0)|_{L^2}^{},|\lapl f|_{L^2}^{};\mu)$ and a constant $c_3(\nu_0)$ such that
\begin{align}
   &|\wt(t)|_{L^2}^2 + \mu\int_t^{t+1} |\gb \wt(\tau)|_{L^2}^2 \dtau \le \eps M_0/\kappa_0^2\label{q:aw}\\
   &|\wt(t)|_{L^2}^2 + \mu\int_0^{t} |\gb \wt(\tau)|_{L^2}^2 \ex^{\nu_0(\tau-t)}\dtau \le \eps M_0/\kappa_0^2\label{q:aw'}
\end{align}
for all $t \ge \TB_0$, where
\begin{equation}\label{q:m0}
   M_0 = c_3\,\Gr_2\Gr_3(1+\Gr_0^2).
\end{equation}
\end{theorem}
\noindent
We note that our $M_0$ is $\kappa_0^2\,M_0$ in \cite{mah-dw:beta}; we have also tightened the bound slightly (this is obvious from the proof), with $\Gr_2\,\Gr_3$ in place of $\Gr_3^2$ in \cite{mah-dw:beta}.

\medskip
For our tighter $\eps$-dependent bounds on the determining modes, it is interesting to consider several forms of zonal forcing often used in numerical simulations of 2d turbulence.
One case is where $\fb$ is bandwidth-limited, in the sense that there is a (modest) $\kpf$ such that
\begin{equation}\label{q:fkpf}
  \fb =\Proj_{\kpf}\fb.
\end{equation}
Another case is where $\fb$ decays exponentially in Fourier space (analytic $\fb$),
\begin{equation}\label{q:fexa}
  |\fb_{(0,k)}^{}| \le \frac{\nu_0^2\,\Gr_0^{}}{2\kpz}\Bigl(\frac{2\alpha}{1+2\alpha}\Bigr)^{1/2}\,\ex^{\alpha(1-|k|/\kpz)},
\end{equation}
where $\alpha>0$.
Finally, we consider algebraically-decaying $\fb$,
\begin{equation}\label{q:fhs}
  |\fb_{(0,k)}^{}| \le \frac{\nu_0^2\kpz^{s-1}\Gr_0}{\sqrt2\zeta(2+2s)^{1/2}}|k|^{-s}
\end{equation}
for $s>5/2$ in order that $\fb\in H^2$.
In both \eqref{q:fexa} and \eqref{q:fhs}, the constants have been chosen so that $|\gb^{-1}\fb|/(\mu\kpz)^2\le\Gr_0$ to be consistent with \eqref{q:gr}.
We stress that no assumptions are made on $\ft$ (other than it being in $H^2$ needed for Theorem~\ref{t:aw}).

Our main result on determining modes follows.

\begin{theorem}\label{t:modes}
Let $\dw$ be the solution of \eqref{q:dwbeta} with $f\in H^2(\Dom)$.
Then the low modes $\Proj_\kappa\,\w$ are determining, i.e.\ 
$\lim_{t\to\infty}|\Proj_\kappa\,\dw(t)|_{L^2}^{}=0$ implies that $\lim_{t\to\infty}|\dw(t)|_{L^2}^{}=0$, if any of the following conditions hold for constants $c_4$, $c_5$, $c_6$ and $\eps$ sufficiently small:\break
\noindent
(a)~if $\fb$ satisfies \eqref{q:fkpf} and
\begin{equation}\label{q:bdkpf}
  \kappa/\kpz > c_4(\nu_0)\,\max\bigl\{ \eps^{1/4} M_0^{1/4},
	\,(\kpf/\kpz)^{3/8}\,\Gr_0^{1/4} \bigr\}; \text{ or}
\end{equation}
(b)~if $\fb$ satisfies \eqref{q:fhs} and
\begin{equation}\label{q:bdhs}
  \kappa/\kpz > c_5(\nu_0,s)\,\max\bigl\{ \eps^{1/4}M_0^{1/4},\,
	\, \Gr_0^{\,(2s+5)/(8s+14)} \bigr\}; \text{ or}
\end{equation}
(c)~if $\fb$ satisfies \eqref{q:fexa} and
\begin{equation}\label{q:bdexa}
  \kappa/\kpz > c_6(\nu_0)\,\max\bigl\{ \eps^{1/4} M_0^{1/4},
	\,F_\alpha\bigl(\nu_0^{-1/2}\Gr_0\bigr)^{3/8}\Gr_0^{1/4} \bigr\}
\end{equation}
where $F_\alpha$ is defined in \eqref{q:falp} below.
\end{theorem}

\noindent
We note that for large $u$, $F_\alpha(u)=\log u/(2\alpha)+\cdots$, so the last term scales essentially as $\Gr_0^{1/4}$.
The smallness requirement on $\eps$, \eqref{q:eps-kf}, \eqref{q:eps-hs} and \eqref{q:eps-exa} below, is not essential and can be removed at the expense of more messy expressions for the above bounds.

As is apparent from the proof below, heuristically one may regard the $\eps^{1/4}M_0^{1/4}$ in \eqref{q:bdkpf} and \eqref{q:bdexa} as arising from the non-zonal forcing $\ft$ and the term involving $\Gr_0$ as arising from the zonal forcing $\fb$.
That the latter bound scales essentially as $\Gr_0^{1/4}$ as opposed to $\Gr_0^{1/2}$ in the general (``non-rotating'') case suggests that, in the limit of small $\eps$, the differentially rotating NSE \eqref{q:dwbeta} essentially consists of a one-dimensional ``mean'' plus a small amount of two-dimensional ``noise'', which agrees with what one would expect on physical grounds.
Barring the discovery of yet unforeseen cancellations, it is therefore unlikely that one could obtain a bound with a smaller power of $\Gr_0$ than $\frac14$.
Similar considerations apply to \eqref{q:bdhs}, where since $\fb\in H^2$ by hypothesis, one must take $s > \sfrac52$, giving a limiting worst-case dependence of $\Gr_0^{\,5/17}$.

\medskip
Analogous to Theorem \ref{t:modes}, we have the following bounds on determining nodes:

\begin{theorem}\label{t:nodes}
Let $\dw$ be the solution of \eqref{q:dwbeta} with $f\in H^2(\Dom)$.
Then there exists a set of determining nodes $\DN=\{\xb_1,\cdots,\xb_N\}$ whenever
\begin{align}
  &N > c_7(\nu_0)\,\max\bigl\{ \eps^{1/2}M_0^{1/2}, (\kpf/\kpz)^{1/3}\,\Gr_0^{2/3} \bigr\}
	\text{ when $\fb$ satisfies \eqref{q:fkpf}; or}\label{q:nkpf}\\
  &N > c_8(\nu_0,s)\,\max\bigl\{ \eps^{1/2}M_0^{1/2}, \,\Gr_0^{\,(4s+5)/(6s+5)}\bigr\}
	\text{ when $\fb$ satisfies \eqref{q:fhs}; or}\label{q:nhs}\\
  &N > c_9(\nu_0)\,\max\bigl\{ \eps^{1/2}M_0^{1/2}, \,F_{\alpha}(\nu_0^{-1}\,\Gr_0^{2/3})^{1/3}\,\Gr_0^{2/3} \bigr\}
	\text{ when $\fb$ satisfies \eqref{q:fexa}},\label{q:nexa}
\end{align}
for constants $c_7$, $c_8$ and $c_9$, $F_{\alpha}$ defined in \eqref{q:nfalp} below and $\eps\le c\,\nu_0/M_0$.
\end{theorem}

\noindent
These nodal results are weaker than their modal counterparts, with the ``zonal part'' scaling essentially as $\Gr_0^{2/3}$ rather than $\Gr_0^{1/2}$.
We believe that this is an artefact of our approach and not intrinsic to the problem.
As in the modal case, the smallness requirement for $\eps$ is not essential and can be removed in exchange for messier expressions in the above bounds.



\section{Proof: Determining Modes}\label{s:pf-modes}

This section is devoted to proving Theorem~\ref{t:modes} using more refined estimates of the nonlinear terms and of the zonal vorticity $\wb$.
For conciseness, when no ambiguity may arise, we write $|\cdot|_p^{}:=|\cdot|_{L^p}^{}$, $|\cdot|:=|\cdot|_{L^2}^{}$ and $(\cdot,\cdot):=(\cdot,\cdot)_{L^2}^{}$.
As usual, $c$ denotes a dimensionless constant whose value may differ in each use.
We also assume for convenience that $\eps\le1$.

First, we collect some basic inequalities.
From the Fourier expansion, we have the following ``improved'' and ``reverse'' Poincar\'e inequalities:
\begin{align}
   &\kappa\,|\dwm>|_{2}^{} \,\le |\gb\dwm>|_{2}^{}\label{q:arrPoi1}\\
   &|\gb\dwm<|_{2}^{} \le \kappa\,|\dwm<|_{2}^{}.\label{q:arrPoi2}
\end{align}
Next, we recall Agmon's inequality in 2d,
\begin{equation}\label{q:ag2}
    |u|_{\infty}^{} \le c\,|u|_2^{1/2}|\lapl u|_2^{1/2}
\end{equation}
for $u \in H^2(\Dom)$.
For functions depending on $y$ and $t$ only, we have the improved version (with the $L^p$ norms always taken over $\Dom$),
\begin{equation}\label{q:ag1}
    |\bar{v}|_{\infty}^{} \le c\,\kpz^{1/2}|\bar{v}|_2^{1/2}|\gb\bar{v}|_2^{1/2}.
\end{equation}

We note the following integral inequality:
Let $\nu>0$ be fixed and $u(t)\ge0$, and suppose that for any $t\ge1$
\begin{equation}
   \int_0^t u(\tau)\,\ex^{\nu(\tau-t)}\dtau \le M,
\end{equation}
then for any $t>0$,
\begin{equation}\label{q:intineq1}
   \int_t^{t+1} u(\tau) \dtau
	\le \int_t^{t+1} \ex^{\nu(\tau-t)} u(\tau)\dtau
	\le \int_0^{t+1} \ex^{\nu(\tau-t)} u(\tau)\dtau
	\le \ex^\nu M.
\end{equation}
Next, we quote the following Gronwall-type lemma from \cite{foias-manley-temam-treve:83,jones-titi:93}.
\begin{lemma}\label{t:gronwall}
Let $\alpha$ and $\beta$ be locally integrable functions on $(0,\infty)$ satisfying
\begin{equation}\begin{aligned}
   &\liminf_{t\to\infty}\int_t^{t+1} \alpha(\tau)\dtau > 0,
   \hbox to 30pt{} \limsup_{t\to\infty}\int_t^{t+1} \alpha^{-}(\tau)\dtau < \infty,\\
   &\lim_{t\to\infty}\int_t^{t+1} \beta^{+}(\tau)\dtau = 0,
\end{aligned}\end{equation}
where $\alpha^{-} := \max\{-\alpha,0\}$ and $\beta^{+} := \max\{\beta,0\}$.
Suppose $\xi$ is an absolutely continuous non-negative function on $(0,\infty)$ such that
\begin{equation}
   \ddt{\xi} + \alpha\xi \leq \beta
\end{equation}
almost everywhere.
Then $\xi(t) \to 0$ as $t \to \infty$.
\end{lemma}

We first use the bound \eqref{q:aw} on the non-zonal $\wt$ to derive a useful control on the {\em zonal\/} vorticity $\wb$.
Fixing some $\kpf\ge\kpz$, let $\wfb{>f}=(1-\Proj_{\kpf})\wb$ and $\fb^{\sst{>f}}=(1-\Proj_{\kpf})\fb$.
We multiply \eqref{q:dwdt} by $\wfb{>f}$ in $L^2$ and compute
\begin{align*}
  \frac12\ddt{}|\wfb{>f}|^2 &+ \mu|\gb\wfb{>f}|^2
   = - (\dy(\psi,\w),\wfb{>f}) + (f,\wfb{>f})\\
	&= - (\dy(\pt,\wt),\wfb{>f}) + (\fb^{\sst{>f}},\wfb{>f}) \hspace{90pt} \text{by \eqref{q:bb0}}\\
	&\le |\gb\pt|_{\infty}^{} |\wt|_2^{}|\gb\wfb{>f}|_2^{} + \frac{2}{\mu}|\gb^{-1}\ffb{>f}|^2 + \frac{\mu}{8}|\gb\wfb{>f}|^2\\
	&\le \frac{2}{\mu}|\gb\pt|_\infty^2|\wt|^2 + \frac{2}{\mu}|\gb^{-1}\ffb{>f}|^2 + \frac{\mu}{4}|\gb\wfb{>f}|^2\\
	&\le \frac{c}{\nu_0}\,\eps M_0 |\gb\pt|\,|\gb\wt| + \frac{2}{\mu}|\gb^{-1}\ffb{>f}|^2 + \frac{\mu}{4}|\gb\wfb{>f}|^2\\
	&\le \frac{c\,\eps M_0}{\nu_0\,\kpz^2} |\gb\wt|^2 + \frac{2}{\mu}|\gb^{-1}\ffb{>f}|^2 + \frac{\mu}{4}|\gb\wfb{>f}|^2
\end{align*}
where we have used \eqref{q:aw} and \eqref{q:ag2} for the penultimate line.
This gives
\begin{equation}\label{q:genwfb}
   \ddt{}|\wfb{>f}|^2 + \frac{3}{2}\,\mu|\gb\wfb{>f}|^2 \le \frac{c\,\eps M_0}{\nu_0\,\kpz^2}|\gb\wt|^2 + \frac{4}{\mu}|\gb^{-1}\ffb{>f}|^2.
\end{equation}
Using Poincar\'e on the lhs and multiplying by $\ex^{\nu_0 t}$, this gives us
\begin{equation}
   \ddt{}\bigl(\ex^{\nu_0 t}|\wfb{>f}|^2\bigr) + \frac\mu2\ex^{\nu_0 t}|\gb\wfb{>f}|^2
	\le \frac{c\,\eps M_0}{\nu_0\,\kpz^2}|\gb\wt|^2\ex^{\nu_0 t} + \frac{4\,\ex^{\nu_0 t}}{\mu}|\gb^{-1}\ffb{>f}|^2,
\end{equation}
and, upon integration in time,
\begin{equation}\label{q:wfbt}\begin{aligned}
  |\wfb{>f}(t)|^2 &+ \frac\mu2 \int_0^t \ex^{\nu_0(\tau-t)}|\gb\wfb{>f}|^2 \dtau\\
	&\le \ex^{-\nu_0t}|\wfb{>f}(0)|^2 + \frac{c\,\eps M_0}{\nu_0\,\kpz^2}\int_0^t |\gb\wt|^2\ex^{\nu_0(\tau-t)}\dtau + \frac{4}{\mu\nu_0}|\gb^{-1}\ffb{>f}|^2\\
	&\le \frac{c_*\,\eps^2 M_0^2}{2\nu_0^2\,\kpz^2} + \frac{4}{\mu\nu_0}|\gb^{-1}\ffb{>f}|^2
\end{aligned}\end{equation}
where we have used \eqref{q:aw'} for the last line, taken $t$ sufficiently large and adjusted the constant.

We now consider the consequences of the hypotheses \eqref{q:fkpf}--\eqref{q:fhs}.
First, when $\fb$ satisfies \eqref{q:fkpf}, we have $\fb^{\sst{>f}}=0$, giving, using \eqref{q:intineq1} and the fact that $\ex^{\nu_0}<3$,
\begin{equation}\label{q:bd-wbfk}
   \int_t^{t+1} |\gb\wb^{\sst{>f}}(\tau)|_{L^2}^2 \dtau
	\le 3c_*\eps^2M_0^2/\nu_0^3.
\end{equation}
Next, for $\fb$ satisfying \eqref{q:fhs}, we have the bound
\begin{equation}
   |\gb^{-1}\ffb{>f}|^2
	\le \frac{\nu_0^4(\kpz/\kpf)^{2s+1}}{(2s+1)\zeta(2s+2)}\frac{\Gr_0^2}{\kpz^4}.
\end{equation}
Using this in \eqref{q:wfbt} and dropping $|\wfb{>f}(t)|^2$ on the lhs gives
\begin{equation}\label{q:bd-wbfs}
   \int_0^t \ex^{\nu_0(\tau-t)}|\gb\wfb{>f}|^2 \dtau
	\le \frac{c_*\eps^2M_0^2}{\nu_0^3} + \frac{8\,(\kpz/\kpf)^{2s+1}\nu_0}{(2s+1)\,\zeta(2s+2)}\,{\Gr_0^2},
\end{equation}
Finally, when $\fb$ satisfies \eqref{q:fexa}, we have
\begin{equation}
   |\gb^{-1}\ffb{>f}|^2
	\le \nu_0^4\frac{2\alpha}{1+2\alpha}\frac{\ex^{2\alpha(1-\kpf/\kpz)}}{1-\ex^{-2\alpha}}\frac{\Gr_0^2}{\kpz^4}
	\le \frac{\nu_0^4}{\kpz^4}\ex^{2\alpha(1-\kpf/\kpz)}\Gr_0^2.
\end{equation}
Using this in \eqref{q:wfbt} and dropping $|\wfb{>f}(t)|^2$ on the lhs as before gives
\begin{equation}\begin{aligned}\label{q:bd-wbfx}
   \int_0^t \ex^{\nu_0(\tau-t)}|\gb\wfb{>f}|^2 \dtau
	&\le \frac{c_*\eps^2M_0^2}{\nu_0^3} + 8\nu_0\ex^{2\alpha(1-\kpf/\kpz)}\Gr_0^2.
\end{aligned}\end{equation}
For both \eqref{q:bd-wbfs} and \eqref{q:bd-wbfx}, suitable $\kpf$ will be chosen when these inequalities are used below in the proof of Theorem \ref{t:modes}.

\begin{proof}[Proof of Theorem \ref{t:modes}]
We multiply \eqref{q:dwbeta} by $\dwm>$ in $L^2$ to obtain
\begin{equation*}\begin{aligned}
   (\dy_t\dw,\dwm>) + (\dy(\pss,\dw),\dwm>) +(\dy(\dps,\w),\dwm>) &+ \frac{\kpz}{\eps}(\dy_x\dps,\dwm>)\\
   &= (\mu\lapl\dw, \dwm>).
\end{aligned}\end{equation*}
Integration by parts shows that the $\kpz/\eps$ term is 0, so
\begin{equation}\label{q:dwbeta2}\begin{aligned}
    \frac{1}{2}\ddt{}|\dwm>|_2^2 &+ \mu|\gb\dwm>|_2^2\\
     &= -(\dy(\pss,\dw),\dwm>) - (\dy(\dpsm<,\w),\dwm>) - (\dy(\dpsm>,\w),\dwm>).
\end{aligned}\end{equation}
For the first term on the right hand side, we use the fact that $(\dy(\pss,\dwm>),\dwm>) = 0$, to get
\begin{equation*}
   (\dy(\pss,\dw),\dwm>) = (\dy(\pss,\dwm<),\dwm>).
\end{equation*}
As for the third term on the right hand side of \eqref{q:dwbeta2}, we write $\w = \wb + \wt$ to get
\begin{equation}\label{q:psszonal}
   (\dy(\dpsm>,\w),\dwm>) = (\dy(\dpsm>, \wt),\dwm>) + (\dy(\dpsm>,\wb), \dwm>),
\end{equation}
the last term of which becomes, by \eqref{q:bb0},
\begin{equation*}
   (\dy(\dpsm>,\wb),\dwm>) = (\dy(\dpstm>,\wb),\dwtm>).
\end{equation*}
For some $\kpf \geq \kpz$ to be fixed later, we split $\wb = \wfb{<f} + \wfb{>f}$, where $\wfb{<f} = \Proj_{\kpf}\wb$ and $\wfb{>f} = \wb-\wfb{<f}$.
Then
\begin{equation}\label{q:midpt}
	(\dy(\dpstm>,\wb),\dwtm>) = (\dy(\dpstm>,\wfb{<f}),\dwtm>) + (\dy(\dpstm>,\wfb{>f}),\dwtm>).
\end{equation}
Thus \eqref{q:dwbeta2} becomes
\begin{equation}\begin{aligned}\label{q:dwbeta3}
	\frac{1}{2}\ddt{}|\dwm>|^2 &+ \mu|\gb\dwm>|^2 \\
		= &-(\dy(\pss,\dwm<),\dwm>) - (\dy(\dpsm<,\w),\dwm>) - (\dy(\dpsm>,\wt),\dwm>)\\
		&- (\dy(\dpstm>,\wfb{<f}),\dwtm>) - \dy(\dpstm>,\wfb{>f}),\dwtm>).
\end{aligned}\end{equation}

We bound the first two terms on the right hand side (recall $|\cdot|_p^{} := |\cdot|_{L^p}^{}$) by
\begin{align}
   |(\dy(\pss,\dwm<),\dwm>)| &\leq |\gb\pss|_\infty^{} |\gb\dwm>|_2^{} |\dwm<|_2^{}\notag\\
	&\leq \frac{4}{\mu}|\gb\pss|_\infty^2|\dwm<|_2^2 + \frac{\mu}{16}|\gb\dwm>|_2^2,\label{q:modes1}\\
   |(\dy(\dpsm<,\w),\dwm>)|
	&\leq |\gb\dwm>|_2^{}\,|\gb\dpsm<|_\infty^{} |\w|_2^{}\notag\\
	&\leq \frac{4}{\mu}|\gb\dpsm<|_\infty^2|\w|_2^2 + \frac{\mu}{16}|\gb\dwm>|_2^2.\label{q:modes2}
\end{align}
We then bound the third term by
\begin{align}
   |(\dy(\dpsm>,\wt),\dwm>)|
	&\le |\gb\dpsm>|_\infty^{} |\gb\wt|_2^{}|\dwm>|_2^{}\notag\\
	&\le c\,|\gb\dpsm>|_2^{1/2}|\gb\dwm>|_2^{1/2}|\dwm>|_2^{}|\gb\wt|_2^{} &&\text{by \eqref{q:ag2}}\notag\\
	&\le \frac{c}{\kappa}\, |\gb\dwm>|_2^{}|\dwm>|_2^{}|\gb\wt|_2^{} &&\text{by \eqref{q:arrPoi1}}\notag\\
	&\le \frac{\mu}{16}|\gb\dwm>|^2 + \frac{c}{\mu\kappa^2}|\gb\wt|^2|\dwm>|^2,\label{q:modes3}
\end{align}
and the fourth term by
\begin{align}
    |(\dy(\dpsm>,\wfb{<f}),\,&\dwm>)| = |(\dy(\dpsm>,\gb\wfb{<f}),\gb\dpsm>)|\notag\\
	&\le c\,|\lapl\wfb{<f}|_\infty^{}|\gb\dpsm>|_2^2\notag\\
	&\le c\,\frac{\kpz^{1/2}}{\kappa^2}|\lapl\wfb{<f}|^{1/2}|\gb^3\wfb{<f}|^{1/2}|\dwm>|^2 &&\text{by \eqref{q:arrPoi1} and \eqref{q:ag1}}\notag\\
	&\le c\,\frac{(\kpz\kpf^3)^{1/2}}{\kappa^3}|\gb\wfb{<f}|\,|\dwm>|\,|\gb\dwm>| &&\text{by \eqref{q:arrPoi2}}\notag\\
	&\le \frac{\mu}{16}|\gb\dwm>|^2 + \frac{c\,\kpz\kpf^3}{\mu\kappa^6}|\gb\omega|^2|\dwm>|^2.\label{q:modes4}
\end{align}
Finally, the last term on the rhs of \eqref{q:dwbeta3} can be bounded as
\begin{align}
   |(\dy(\dpsm>,\wfb{>f}),\,&\dwm>)| \leq |\gb\dpsm>|_2^{}|\wfb{>f}|_{\infty}^{}|\gb\dwm>|_2^{}\notag\\
	&\le c\,\frac{\kpz^{1/2}}{\kappa}|\wfb{>f}|^{1/2}|\gb\wfb{>f}|^{1/2}|\gb\dwm>|\,|\dwm>| &&\text{by \eqref{q:arrPoi1} and \eqref{q:ag1}}\notag\\
	&\le c\,\frac{\kpz^{1/2}}{\kappa\kpf^{1/2}}|\gb\wfb{>f}|\,|\gb\dwm>|\,|\dwm>| &&\text{by \eqref{q:arrPoi1}}\notag\\
	&\le \frac\mu{16}|\gb\dwm>|^2 + \frac{c\kpz}{\mu\kappa^2\kpf}|\gb\wfb{>f}|^2|\dwm>|^2.\label{q:modes5}
\end{align}
Putting all these together and applying \eqref{q:arrPoi1} on the lhs gives
\begin{align*}
	\ddt{}|\dwm>|^2 &+ |\dwm>|^2\Bigl(\mu\kappa^2 - \frac{c}{\mu\kappa^2}\,|\gb\wt|^2 - \frac{c\kpz\kpf^3}{\mu\kappa^6}|\gb\omega|^2 - \frac{c\kpz}{\mu\kappa^2\kpf}|\gb\wfb{>f}|^2\Bigr)\\
	&\leq \frac{8}{\mu}|\gb\pss|_\infty^2|\dwm<|^2 + \frac{8}{\mu}|\gb\dpsm<|_\infty^2|\w|^2.
\end{align*}
We aim to apply Lemma \ref{t:gronwall} to this, with $\xi = |\dwm>|^2$, $\alpha$ the large bracket on the lhs and $\beta$ the rhs.
Now the hypothesis of the lemma on $\beta$ is satisfied since $|\dwm<(t)|\to0$ as $t\to\infty$ (and what multiply it are bounded when integrated in time), and that on $\xi$ follows from the standard regularity of the NSE.
The hypothesis on $\alpha$ would follow from
\begin{equation}\label{q:col2}
	\limsup_{t\to\infty}\int_t^{t+1} \Bigl(\frac1{\mu\kappa^2}|\gb\wt|^2 + \frac{\kpz\kpf^3}{\mu\kappa^6}|\gb\omega|^2 + \frac{\kpz}{\mu\kappa^2\kpf}|\gb\wfb{>f}|^2\Bigr) \dtau < c\mu\kappa^2,
\end{equation}
which in turn is implied by the conditions
\begin{align}
   &\limsup_{t\to\infty}\int_t^{t+1} |\gb\wt|^2 \dtau < c\,\mu^2\kappa^4,\label{q:ci1}\\
   &\limsup_{t\to\infty}\int_t^{t+1} |\gb\omega|^2 \dtau < \frac{c\,\mu^2\kappa^8}{\kpz\kpf^3},\label{q:ci2}\\
   &\limsup_{t\to\infty}\int_t^{t+1} |\gb\wfb{>f}|^2 \dtau < \frac{c\,\mu^2\kappa^4\kpf}{\kpz}.\label{q:ci3}
\end{align}
For the first condition, we note that \eqref{q:aw} implies
\begin{align*}
   \int_t^{t+1} |\gb\wt|^2 \dtau
	\le {\eps M_0}/{\nu_0}
\end{align*}
so \eqref{q:ci1} would follow if
\begin{equation}\label{q:epbd}
   \kappa/\kpz > c\,(\eps M_0/\nu_0^3)^{1/4}.
\end{equation}
By \eqref{q:wgm}, the second condition is implied by
\begin{equation}\label{q:bdk-pf1}
   c\,\Gr_0^2\nu_0 < \mu^2\kappa^8/(\kpz\kpf^3)
   \quad\Leftrightarrow\quad
   \kappa/\kpz > c\,\nu_0^{-1/8}(\kpf/\kpz)^{3/8}\,\Gr_0^{1/4}.
\end{equation}

\medskip
We first consider the case when $\fb$ satisfies \eqref{q:fkpf}.
By \eqref{q:bd-wbfk}, we have
\begin{equation*}
   \int_t^{t+1} |\gb\wfb{>f}|^2 \dtau
	\le c\,\eps^2M_0^2/\nu_0^3
\end{equation*}
so in this case \eqref{q:ci3} would hold if
\begin{equation}\label{q:bdk-pf2}
   \kappa/\kpz > c\,(\eps M_0)^{1/2}\nu_0^{-5/4}(\kpz/\kpf)^{1/4}.
\end{equation}
This bound is dominated by \eqref{q:epbd} when
\begin{equation}\label{q:eps-kf}
   \eps M_0 \le c\,\nu_0^2(\kpf/\kpz),
\end{equation}
which we hereby assume.
Combining \eqref{q:epbd}, \eqref{q:bdk-pf1} and \eqref{q:bdk-pf2}, we recover \eqref{q:bdkpf}.

\medskip
Next, we consider the case when $\fb$ satisfies \eqref{q:fhs}.
By \eqref{q:intineq1} and \eqref{q:bd-wbfs}, we have
\begin{equation}\label{q:pwk-pf1}
   \int_t^{t+1} |\gb\wfb{>f}|^2 \dtau
	\le c\,\eps^2M_0^2/\nu_0^3 + c\,c_\zeta(s)\,\nu_0\,(\kpz/\kpf)^{2s+1}\,\Gr_0^2 =: I_1
\end{equation}
where $1/c_\zeta(s):=(2s+1)\,\zeta(2s+2)$.
So \eqref{q:ci3} would be satisfied if $I_1<c\mu^2\kappa^4(\kpf/\kpz)$;
analogously to what we did with \eqref{q:col2}, this is in turn implied by the inequalities
\begin{align}
   (\kappa/\kpz)^4 &> c\,(\eps M_0)^2\nu_0^{-5}(\kpz/\kpf)\label{q:ci3s1}\\
   (\kappa/\kpz)^4 &> c\,c_\zeta(s)\,\nu_0^{-1}(\kpz/\kpf)^{2s+2}\,\Gr_0^2.\label{q:ci3s2}
\end{align}
Since both \eqref{q:bdk-pf1} and \eqref{q:ci3s2} must be satisfied, we equate these bounds and find
\begin{equation}\label{q:pwk-kpf}
   (\kpf/\kpz)^{2s+7/2} = c\,c_\zeta(s)\nu_0^{-1/2}\Gr_0,
\end{equation}
which fixes $\kpf$ and turns both \eqref{q:bdk-pf1} and \eqref{q:ci3s2} to
\begin{equation}\label{q:pwk-pf4}
	\kappa/\kpz > c\,\bigl(c_{\zeta}(s)^{3/2}\,\nu_0^{-(s+5/2)}\,\Gr_0^{2s+5}\bigr)^{1/(8s+14)}.
\end{equation}
Using \eqref{q:pwk-kpf}, \eqref{q:ci3s1} becomes
\begin{equation}
   \kappa/\kpz > c_s\,(\eps\,M_0)^{1/2} \nu_0^{-5/4+1/(16s+28)}\Gr_0^{-1/(8s+14)}
\end{equation}
with $c_s=c\,c_\zeta(s)^{-1/(8s+14)}$, noting that since $s>5/2$ the exponent for $\Gr_0$ lies between $-1/34$ and $0$, giving a weak dependence.
This term is dominated by \eqref{q:epbd} when
\begin{equation}\label{q:eps-hs}
   \eps\,M_0 \le c\,c_s^{-4}\,\nu_0^{2-1/(4s+7)}\Gr_0^{2/(4s+7)}.
\end{equation}
Assuming this, \eqref{q:bdhs} follows from \eqref{q:epbd} and \eqref{q:pwk-pf4}.


\medskip
Finally we consider $\fb$ satisfying \eqref{q:fexa}.
By \eqref{q:intineq1} and \eqref{q:bd-wbfx}, we have
\begin{equation}\label{q:exk-pf1}
   \int_t^{t+1} |\gb\wfb{>f}|^2\dtau
	\le c\,(\eps M_0)^2/\nu_0^3 + c\,\nu_0 \,\ex^{2\alpha(1-\kpf/\kpz)}\Gr_0^2.
\end{equation}
As before, \eqref{q:ci3} would be satisfied if both of the following hold:
\begin{align}
   &(\kappa/\kpz)^4 > c\,(\eps M_0)^2\nu_0^{-5}(\kpz/\kpf)\label{q:ci3a1}\\
   &(\kappa/\kpz)^4 > c\,\nu_0^{-1}\ex^{2\alpha(1-\kpf/\kpz)}(\kpz/\kpf)\,\Gr_0^2.\label{q:ci3a2}
\end{align}
Equating the bounds from \eqref{q:bdk-pf1} and \eqref{q:ci3a2}, we arrive at
\begin{equation}
   (\kpf/\kpz)^{5/2}\ex^{2\alpha(\kpf/\kpz-1)} = c_\alpha\,\nu_0^{-1/2}\,\Gr_0,
\end{equation}
which can be inverted to give
\begin{equation}\label{q:falp}
   \kpf/\kpz = F_\alpha\bigl(\Gr_0/\nu_0^{1/2}\bigr)
   \quad\text{where}\quad
   F_\alpha^{-1}(y)=y^{5/2}\ex^{2\alpha (y-1)}/c_\alpha.
\end{equation}
This fixes $\kpf$.
Now the bound from \eqref{q:epbd} would dominate that from \eqref{q:ci3a1} when
\begin{equation}\label{q:eps-exa}
   \eps M_0 \le c\,\nu_0^2\,(\kpf/\kpz).
\end{equation}
Assuming this, \eqref{q:bdexa} follows from \eqref{q:epbd} and \eqref{q:bdk-pf1}.
\end{proof}



\section{Proof: Determining Nodes}\label{s:pf-nodes}

In this section, we prove Theorem~\ref{t:nodes}. We follow the notations and conventions of \S\ref{s:pf-modes}.
We use several crucial inequalities proved in \cite{jones-titi:93}.
Following them, let the points $\xb_1,\cdots,\xb_N$ be placed at regular spacings within our periodic domain $\Dom=[0,L]\times[-L/2,L/2]$.
Defining
\begin{equation}\label{q:defeta}
	\eta(u) := \max_{1\leq i \leq N}|u(\xb_i)|
\end{equation}
for all $u \in H^2(\Dom)$ satisfying \eqref{q:v0int},
we have the following bounds:
\begin{equation}\label{q:jtnodes1}
	|u|_{L^2}^2 \hspace{37pt} \leq c_\eta \,L^2\eta(u)^2 + c_\eta \frac{L^4}{N^2}|\lapl u|_{L^2}^2,
\end{equation}
\begin{equation}\label{q:jtnodes2}
	|\gb u|_{L^2}^2, |u|_{L^\infty}^2 \leq c_\eta\, N \eta(u)^2 + c_\eta \frac{L^2}{N}|\lapl u|_{L^2}^2,
\end{equation}
for an absolute constant $c_\eta$.

\begin{proof}[Proof of Theorem \ref{t:nodes}]
We multiply \eqref{q:dwbeta} by $\dw$ in $L^2$ to obtain
\begin{equation*}
	(\dy_t\dw,\dw) + (\dy(\pss,\dw),\dw) + (\dy(\dps,\w),\dw) + \frac{\kpz}{\eps}(\dy_x\dps,\dw) = \mu(\lapl\dw,\dw).
\end{equation*}
The second and fourth term vanish upon integration by parts, giving
\begin{equation}\label{q:nodes0}
	\frac{1}{2}\ddt{}|\dw|^2 + \mu|\gb\dw|^2 = -(\dy(\dps,\w),\dw).
\end{equation}
We use \eqref{q:bb0} and the splitting $\w = \wb + \wt$ to write the rhs as
\begin{equation}
	-(\dy(\dps,\w),\dw) = -(\dy(\dps,\wt),\dw) - (\dy(\dps,\wb),\dw).
\end{equation}
As in \eqref{q:midpt}, we split $\wb = \wfb{<f} + \wfb{>f}$ where $\wfb{<f} = \Proj_{\kpf}\wb$ and $\wfb{>f} = \wb-\wfb{<f}$, for some $\kpf \geq \kpz$ to be fixed later.
Now \eqref{q:nodes0} becomes
\begin{equation}\label{q:dwnodes}
	\frac{1}{2}\ddt{}|\dw|^2 + \mu|\gb\dw|^2 = -(\dy(\dps,\wt),\dw) - (\dy(\dps,\wfb{<f}),\dw) - (\dy(\dps,\wfb{>f}),\dw).
\end{equation}

For $\DN$, we pick $N$ equally spaced points $\{\xb_1,\cdots,\xb_N\}$.
We bound the first term on the rhs using \eqref{q:jtnodes2},
\begin{equation}\label{q:nodes1}\begin{aligned}
   \bigl|(\dy(\dps,\wt),\dw)\bigr|
	&\le |\gb\dps|_\infty^{}|\gb\wt|_2^{}|\dw|_2^{}\\
	&\le \frac{c\mu N}{L^2}|\gb\dps|_\infty^2 + \frac{cL^2}{\mu N}|\gb\wt|^2|\dw|^2\\
	&\le \frac{c\mu N}{L^2}\Bigl[N\eta(\gb\dps)^2 + \frac{L^2}{N}|\gb\dw|^2\Bigr] + \frac{cL^2}{\mu N}|\gb\wt|^2|\dw|^2.
\end{aligned}\end{equation}
Similarly, we bound the second term on the rhs of \eqref{q:dwnodes} using Young and \eqref{q:jtnodes1} as
\begin{equation}\label{q:nodes2}\begin{aligned}
   \bigl|(\dy(\dps,\wfb{<f}),\,&\dw)\bigr|
	\le |\gb\wfb{<f}|_\infty^{}|\gb\dps|_2^{}|\dw|_2^{}\\
	&\le c(\kpz\kpf)^{1/2}|\gb\omega|\,|\gb\dps|\,|\dw|\\
	&\le \frac{c\mu N^2}{L^4}|\gb\dps|^2 + \frac{cL^4}{\mu N^2}\kpz\kpf|\gb\omega|^2|\dw|^2\\
	&\le \frac{c\mu N^2}{L^4}\Bigl[L^2\eta(\gb\dps)^2 + \frac{L^4}{N^2}|\gb\dw|^2\Bigr] + \frac{cL^4}{\mu N^2}\kpz\kpf|\gb\omega|^2|\dw|^2.
\end{aligned}\end{equation}
The final term in \eqref{q:dwnodes} we bound as
\begin{equation}\label{q:nodes3}\begin{aligned}
   \bigl|(\dy(\dps,\wfb{>f}),\,&\dw)\bigr|
   \le |\gb\dps|_\infty^{}|\gb\wfb{>f}|_2^{}|\dw|_2^{}\\
   &\le \frac{c\mu N}{L^2}|\gb\dps|_{\infty}^2 + \frac{cL^2}{\mu N}|\gb\wfb{>f}|^2|\dw|^2\\
   &\le \frac{c\mu N}{L^2}\Bigl[N\eta(\gb\dps)^2 + \frac{L^2}{N}|\gb\dw|^2\Bigr] + \frac{cL^2}{\mu N}|\gb\wfb{>f}|^2|\dw|^2.
\end{aligned}\end{equation}
Applying \eqref{q:jtnodes2} to the lhs of \eqref{q:dwnodes} as
\begin{equation}
   \frac12\ddt{\;}|\dw|^2 + \frac{c\mu N}{L^2}|\dw|^2 - \frac{c\mu N^2}{L^2}\eta(\gb\dps)^2
	\le \frac12\ddt{\;}|\dw|^2 + \mu|\gb\dw|^2,
\end{equation}
and putting together \eqref{q:nodes1}--\eqref{q:nodes3} gives, after some rearrangement,
\begin{equation}\label{q:ncol0}\begin{aligned}
    \ddt{\;}|\dw|^2 &+ |\dw|^2\Bigl[\frac{c\mu N}{L^2} - \frac{cL^2}{\mu N}|\gb\wt|^2 - \frac{cL^4}{\mu N^2}\kpz\kpf|\gb\omega|^2 - \frac{cL^2}{\mu N}|\gb\wfb{>f}|^2\Bigr]\\
	&\le \frac{c\mu N^2}{L^2}\eta(\gb\dps)^2.
\end{aligned}\end{equation}

As in the proof of Theorem \ref{t:modes}, we aim to apply Lemma \ref{t:gronwall} to $\xi = |\dw|^2$, $\alpha$ being the large bracket on the lhs and $\beta$ the rhs of \eqref{q:ncol0}.
The hypothesis of the lemma on $\beta$ is met because $\gb\dps(\xb_i,t) \to 0$ as $t \to \infty$ for all $i$ and $|\gb\omega|$ is bounded, while the hypothesis on $\xi$ follows from the regularity of the NSE.
The hypothesis on $\alpha$ would follow from, noting that $\nu_0=c\mu/L^2$,
\begin{equation}\label{q:ni}
   \limsup_{t\to\infty}\int_t^{t+1} \Bigl( \frac{c}{\nu_0 N}|\gb\wt|^2 + \frac{c}{\nu_0 N^2}\frac\kpf\kpz|\gb\omega|^2 + \frac{c}{\nu_0N}|\gb\wfb{>f}|^2 \Bigr) \dtau
	< \nu_0 N.
\end{equation}
With no loss of generality, we require that this inequality is satisfied by each term separately (adjusting the $c$ as usual).

For the first term, we note that \eqref{q:aw} implies
\begin{equation}
	\int_t^{t+1} |\gb\wt|^2 \dtau \leq \eps M_0/\nu_0.
\end{equation}
so \eqref{q:ni} for $|\gb\wt|^2$ would be satisfied for
\begin{equation}\label{q:nepbd}
	N^2 > c\,\eps M_0/\nu_0^3.
\end{equation}
For the second term, we have by \eqref{q:wgm}
\begin{equation}
   \int_t^{t+1} |\gb\omega|^2 \dtau \le c\,\nu_0\Gr_0^2,
\end{equation}
so the $|\gb\omega|$ part of \eqref{q:ni} is implied by
\begin{equation}\label{q:n<fbd}
   N > \frac{c}{\nu_0^{1/3}} \Bigl(\frac{\kpf}{\kpz}\Bigr)^{1/3}\Gr_0^{2/3}.
\end{equation}
For the inequality involving $|\gb\wfb{>f}|^2$, we need to handle the cases separately.

\medskip
We consider first when $\fb$ satisfies \eqref{q:fkpf}.
By \eqref{q:bd-wbfk},
\begin{equation}
   \int_t^{t+1} |\gb\wfb{>f}|^2 \dtau
	\le c\,\eps^2 M_0^2/\nu_0^3,
\end{equation}
so the $|\gb\wfb{>f}|$ part of \eqref{q:ni} holds if
\begin{equation}\label{q:nbdk-pf1}
	N > c\,\eps M_0/\nu_0^{5/2}.
\end{equation}
Since $\nu_0\le1$, this bound is dominated by \eqref{q:nepbd} when $\eps M_0\le c\,\nu_0^2$.
Assuming this, \eqref{q:nkpf} follows from \eqref{q:nepbd} and \eqref{q:n<fbd}.

\medskip
For $\fb$ instead satisfying \eqref{q:fhs}, we recall from \eqref{q:pwk-pf1} that
\begin{equation}\label{q:pwk-ni3}
	\int_t^{t+1} |\gb\wfb{>f}|^2 \dtau
	\le c\,\eps^2M_0^2/\nu_0^3 + c\,c_{\zeta}(s)\,\nu_0\,(\kpz/\kpf)^{2s+1}\,\Gr_0^2 = I_1,
\end{equation}
where $1/c_{\zeta}(s) = (2s+1)\,\zeta(2s+2)$.
Therefore, the $|\gb\wfb{>f}|^2$ part of \eqref{q:ni} would be satisfied if $I_1 \leq c\,\nu_0^2 N^2$;
analogously to \eqref{q:pwk-pf1}, this in turn is implied by
\begin{equation}\label{q:ni3s1}
	N^2 > c\,(\eps M_0)^2\nu_0^{-5},
\end{equation}
\begin{equation}\label{q:ni3s2}
	N^2 > c\,c_\zeta(s)\,\nu_0^{-1}(\kpz/\kpf)^{2s+1}\,\Gr_0^2.
\end{equation}
Since \eqref{q:n<fbd} and \eqref{q:ni3s2} must both hold, we equate these bounds to fix $\kpf$:
\begin{equation}\label{q:pwk-nkpf}
	(\kpf/\kpz)^{2s+5/3} = c\,c_\zeta(s)\,\nu_0^{-1}\Gr_0^{2/3},
\end{equation}
with which both \eqref{q:n<fbd} and \eqref{q:ni3s2} now read
\begin{equation}\label{q:pwk-ni2s2}
	N > c\,(c_\zeta(s)\nu_0^{-1}\,\Gr_0^{4s+5})^{1/(6s+5)}.
\end{equation}

As with the case when $\fb$ satisfies \eqref{q:fkpf}, \eqref{q:nhs} follows from \eqref{q:nepbd} and \eqref{q:pwk-ni2s2} when $\eps M_0 \le c\,\nu_0^2$.

Finally we consider $\fb$ satisfying \eqref{q:fexa}.
By \eqref{q:intineq1} and \eqref{q:bd-wbfx},
\begin{align*}
   \int_t^{t+1} |\gb\wfb{>f}|^2 \dtau
	&\leq c\,(\eps M_0)^2/\nu_0^3 + c\,\nu_0\,\ex^{2\alpha(1-\kpf/\kpz)}\,\Gr_0^2.
\end{align*}
As before, the $|\gb\wfb{>f}|^2$ part of \eqref{q:ni} is satisfied when both of the following hold:
\begin{equation}\label{q:ni3a1}
	N^2 > c\,(\eps M_0)^2\nu_0^{-5},
\end{equation}
\begin{equation}\label{q:ni3a2}
	N^2 > c\,\nu_0^{-1}\ex^{2\alpha(1-\kpf/\kpz)}\,\Gr_0^2.
\end{equation}
We equate the rhs of \eqref{q:n<fbd} and \eqref{q:ni3a2} to obtain
\begin{equation*}
	(\kpf/\kpz)^{2/3}\,\ex^{2\alpha(\kpf/\kpz-1)} = c_\alpha\,\nu_0^{-1}\,\Gr_0^{2/3},
\end{equation*}
which we invert to find
\begin{equation}\label{q:nfalp}
	\kpf/\kpz = F_{\alpha}(\nu_0^{-1}\Gr_0^{2/3})
\end{equation}
where $F_{\alpha}^{-1}(y) := y^{2/3}\ex^{2\alpha(y-1)}/c_\alpha$ is that in \eqref{q:falp}, abusing notation slightly (possibly different $c_\alpha$).
As before, assuming $\eps M_0 \le c\,\nu_0^2$, \eqref{q:nexa} follows from \eqref{q:nepbd} and \eqref{q:n<fbd}.
This concludes the proof.
\end{proof}



\nocite{temam:iddsmp}
\nocite{vallis:aofd}

\nocite{rhines:75}
\nocite{tgs:87}

\nocite{schochet:94}
\nocite{gallagher-straymond:07}


\end{document}